\documentclass[12pt]{amsart}

\usepackage{mathrsfs}
\usepackage{amsfonts}
\usepackage{amssymb}
\usepackage{amsfonts, amscd, amsmath, mathrsfs, amssymb, amsthm, amsxtra, bbding, epsfig, graphicx, latexsym, url, mathbbol, bbold}
\usepackage[papersize={8.3in,11.8in},textwidth=16.8cm,textheight=23.2cm,centering]{geometry}
\usepackage{enumerate}
\usepackage{caption} 

\usepackage{graphicx}
\usepackage{tikz}
\usepackage{float}

\usepackage{xcolor}
\definecolor{cite}{rgb}{0.00,0.60,1.00}
\definecolor{url}{rgb}{1.00,0.10,0.80}
\definecolor{link}{rgb}{0.00,0.00,1.00}
\usepackage[colorlinks,linkcolor=link,urlcolor=url,citecolor=cite,pagebackref,breaklinks]{hyperref}

\hypersetup{
pdfstartpage=1,
pdfstartview=FitH}

\DeclareFontFamily{U}{mathx}{\hyphenchar\font45}
\DeclareFontShape{U}{mathx}{m}{n}{
      <5> <6> <7> <8> <9> <10>
      <10.95> <12> <14.4> <17.28> <20.74> <24.88>
      mathx10
      }{}
\DeclareSymbolFont{mathx}{U}{mathx}{m}{n}
\DeclareMathAccent{\widecheck}{\mathalpha}{mathx}{"71}

\numberwithin{equation}{section}

\allowdisplaybreaks

\newtheorem{theorem}{Theorem}[section]
\newtheorem{lemma}[theorem]{Lemma}
\newtheorem{question}[theorem]{Question}
\newtheorem{definition}[theorem]{Definition}

\newtheorem{corollary}[theorem]{Corollary}
\newtheorem{conjecture}[theorem]{Conjecture}

\makeatletter
\newcounter{roem}
\renewcommand{\theroem}{\Roman{roem}}

\newcommand{\c@org@eq}{}
\let\c@org@eq\c@equation
\newcommand{\org@theeq}{}
\let\org@theeq\theequation

\newcommand{\setroem}{
\let\c@equation\c@roem
 \let\theequation\theroem}

\newcommand{\setarab}{
\let\c@equation\c@org@eq
\let\theequation\org@theeq}
\makeatother

\newtheorem*{claim*}{Claim}

\usepackage{scalerel}

\def\leq{\leqslant}

\def\geq{\geqslant}

\usepackage{graphicx}
\usepackage{tikz}

\begin{document}

\title{On almost perfect additive complements}

\author{Wouter van Doorn, Boning Liu, Quanyu Tang}


\begin{abstract} 
Let $A, B$ be infinite sets of non-negative integers with $n\in A+B$ for all sufficiently large integers $n$. We prove that if $A(x)B(x)-x$ does not tend to infinity, then
$$\limsup_{x\to\infty}\frac{A(x)B(x)}{x}\in\Big[\frac{3}{2}, 2\Big].$$
The lower bound is sharp, and confirms a conjecture of Chen. Our proof is based on the weak limits and the convolution factorization theorem of Gabardo and Lai. A discrete proof of the same bound was also given by Zhang. Motivated by the known spectrum of perfect additive complements, we conjecture that the interval $(3/2, 8/5)$ is also excluded for almost perfect additive complements. 
\end{abstract}

\maketitle

\setcounter{tocdepth}{1}



\section{Introduction}

Let $A, B$ be two subsets of non-negative integers, and $A(x), B(x)$ be the corresponding {\it counting functions} defined by
$$A(x):=\sum_{\substack{a\in A\\0\leq a\leq x}}1,\ B(x):=\sum_{\substack{b\in B\\0\leq b\leq x}}1.$$
\begin{itemize}
	\item $A, B$ are called {\it infinite}, if $A, B$ are both infinite sets.
	\item $A, B$ are called {\it additive complements}, if each sufficiently large integer $n$ can be written as $n=a+b$ for $a\in A, b\in B$.
	\item Additive complements $A, B$ are called {\it almost perfect}, if $A(x)B(x)-x\not\to\infty$. 
	\item Additive complements $A, B$ are called {\it perfect}, if the expression $n=a+b$ is unique for all $n\in\mathbb{Z}_{\geq0}$. 
\end{itemize}

For infinite additive complements, S\'ark\"ozy and Szemer\'edi \cite{SS1994} proved Danzer's conjecture\cite{Da1964}:
$$\limsup_{x\to\infty}\frac{A(x)B(x)}{x}\leq1\implies A(x)B(x)-x\to\infty.$$
The problems were studied by many authors \cite{Da1964,EG1980,Na1959,SS1994}. Chen and Fang \cite{CF2011,CF2013,CF2014,CF2020} extended the result to show that for infinite almost perfect additive complements $A, B$,
$$\limsup_{x\to\infty}\frac{A(x)B(x)}{x}\in[3-\sqrt{3}, 2].$$
Yang and Zhao \cite{YZ2026} improved the lower bound to $4/3$ recently. Chen conjectured that the lower bound of the interval should be $3/2$, which is a sharp bound. In this paper, we confirm Chen's conjecture. The proof is shown in Section \ref{section:Chen_proof}.
\begin{theorem}[Chen's conjecture]\label{thm:chen}
	If $A, B$ are infinite almost perfect additive complements, then
	$$\limsup_{x\to\infty}\frac{A(x)B(x)}{x}\in\Big[\frac{3}{2}, 2\Big].$$
\end{theorem}
For perfect additive complements,  B\'ar\'any, Fang and S{\'a}ndor \cite{BFS2023} obtained a detailed description of the {\it spectrum}, including the gap stated below. In Section \ref{section:eg} we give an elementary proof for this gap, which also motivates Conjecture \ref{conj:almost_perfect_gap}.
\begin{theorem}
	If $A, B$ are infinite perfect additive complements, then
	$$\limsup_{x\to\infty}\frac{A(x)B(x)}{x}\in\Big\{\frac{3}{2}\Big\}\cup\Big[\frac{8}{5}, 2\Big].$$
\end{theorem}
\begin{conjecture}\label{conj:almost_perfect_gap}
	If $A, B$ are infinite almost perfect additive complements, then
	$$\limsup_{x\to\infty}\frac{A(x)B(x)}{x}\in\Big\{\frac{3}{2}\Big\}\cup\Big[\frac{8}{5}, 2\Big].$$
\end{conjecture}

\subsection*{Use of AI tools}

We independently proved Theorem \ref{thm:perfect}, which led us to consider a more general conjecture. ChatGPT 5.6 Pro identified the applicability of the result of Gabardo and Lai \cite{GL2014} and generated an initial proof of Chen's conjecture (Theorem \ref{thm:chen}). We checked the proof and gave a refined version. The initial AI-generated proof has also been formalized and verified in Lean. The original proof and its Lean formalization are available in \cite{erdos785-lean}.

\section{Preliminary Lemmas}

Let $r(n)$ be the number of representations of $n$ as $a+b$, defined by
$$r(n)=r_{A, B}(n):=\#\{(a, b)\in A\times B \mid a+b=n\}.$$
We call $(a, b)\in A\times B$ a {\it crossing pair} of level $N$, if $a,b\leq N<a+b$, and let $\mathcal{C}(N)$ be the number of all crossing pairs of level $N$,
$$\mathcal{C}(N)=\mathcal{C}_{A, B}(N):=\#\{(a, b)\in A\times B \mid 0\leq a,b\leq N<a+b\}.$$

The following lemma will provide some basic properties of the infinite almost perfect additive complements. 

\begin{lemma}\label{lm:everything}
	Let $A, B$ be infinite almost perfect additive complements. Let 
	$$L=\liminf_{n\to\infty} (A(n)B(n)-n)>-\infty.$$
	Take an increasing integer sequence $\{N_j\}$ such that $A(N_j)B(N_j)-N_j=L$. Then we have
	
	$(i)$. There exists a constant $n_0$ such that
	\begin{equation}\label{eq:r(n)=1}
		r(n)=1
	\end{equation}
	for all $n>n_0$. 
	
	$(ii)$. There exists a constant $C_0$ such that
	\begin{equation}\label{eq:A(u)B(v)<u+v}
		A(u)B(v)\leq u+v+C_0
	\end{equation}
	for all $u, v\in\mathbb{Z}_{\geq0}$.

	$(iii)$. The number of crossing pairs of level $N_j$ is
	\begin{equation}\label{eq:C(N_j)=O(1)}
		\mathcal{C}(N_j)=O(1), j\to\infty.
	\end{equation}
	Particularly, $N_j\not\in A\cup B$ for all large $N_j$.
\end{lemma}

\begin{proof}
	$(i)$. For additive complements $A, B$ and integer $N$, we have
	$$A(N)B(N)-N=1+\sum_{0\leq n\leq N}(r(n)-1)+\mathcal{C}(N).$$
	Hence for $N=N_j$ we have
	$$L=\sum_{0\leq n\leq N_j}(r(n)-1)+\mathcal{C}(N_j)+1.$$
	Since $\mathcal{C}(N)\geq0$, We have
	$$\sum_{0\leq n\leq N_j}(r(n)-1)\leq L-1$$
	uniformly for all $N_j$, and $r(n)\geq1$ for large $n$. Thus $r(n)-1=0$ for all large $n$.

	$(ii)$. From \eqref{eq:r(n)=1} we have
	$$A(u)B(v)= \sum_{0\leq a\leq u}\sum_{0\leq b\leq v}1\leq \sum_{0\leq n\leq u+v}r(n)\leq \sum_{0\leq n\leq n_0}(n+1)+\sum_{n_0<n\leq u+v}1.$$
	It suffices to take $C_0=n_0^2+114514$.

	$(iii)$. Following (i), the number of crossing pairs
	$$\mathcal{C}(N_j)=L-1-\sum_{0\leq n\leq N_j}(r(n)-1)$$
	is also bounded, and is eventually a constant for $N_j>n_0$.

	Furthermore, suppose $N_j\in A$, then $(N_j, b)$ is a crossing pair for all $b\in B\cap(0, N_j]$, and
	$$\mathcal{C}(N_j)\geq B(N_j)-1\to\infty,$$
	a contradiction. For the same reason we have $N_j\not\in A\cup B$. 
\end{proof}

The next lemma shows the uniqueness of the differences in the sense of the congruence, which is crucial for us to estimate the trivial case. 

\begin{lemma}\label{lm:unique_difference}
	Following the notations in Lemma \ref{lm:everything}. For all but at most $O(1)$ exceptional $n\in\mathbb{Z}\cap[0, N_j)$, there exists a unique pair $(a, b)\in (A\cap[0, N_j))\times(B\cap[0, N_j))$ such that
	$$b-a\equiv n\ (\mathrm{mod}\ N_j).$$
\end{lemma}

\begin{proof}
	Recall that $N_j\not\in A\cup B$ for large $N_j$. For integer $n\in[0, N_j)$, Define
	$$R(n):=\#\{(a, b)\in A\times B\mid 0\leq a,b< N_j,\ a+b\equiv n\ (\mathrm{mod}\ N_j)\},$$
	$$D(n):=\#\{(a, b)\in A\times B\mid 0\leq a,b< N_j,\ b-a\equiv n\ (\mathrm{mod}\ N_j)\}.$$
	We want to show that $D(n)=1$ for almost all $n$. Notice that
	$$\sum_{0\leq n<N_j}R(n)=\sum_{0\leq n<N_j}\sum_{\substack{a\in A,\, b\in B\\0\leq a,b<N_j\\a+b\equiv n(N_j)}}1=\sum_{a, b}1=A(N_j)B(N_j)=\sum_{0\leq n<N_j}D(n),$$
	$$\sum_{0\leq n<N_j}R(n)^2=\sum_{0\leq n<N_j}\sum_{\substack{a_1, a_2, b_1, b_2\\a_1+b_1\equiv a_2+b_2\equiv n(N_j)}}1=\sum_{\substack{a_1, a_2, b_1, b_2\\a_1+b_1\equiv a_2+b_2(N_j)}}1=\sum_{\substack{a_1, a_2, b_1, b_2\\b_1-a_2\equiv b_2-a_1(N_j)}}1=\sum_{0\leq n<N_j}D(n)^2,$$
	where $u\equiv v\ (N_j)$ is an abbreviation for $u\equiv v\ (\mathrm{mod}\ N_j)$. Then we have
	\begin{equation}
		\sum_{0\leq n<N_j}R(n)(R(n)-1)=\sum_{0\leq n<N_j}D(n)(D(n)-1).
	\end{equation}
	To estimate $R(n)$, define
	$$r^{(j)}(n):=\#\{(a, b)\in A\times B\mid 0\leq a,b< N_j,\ a+b=n\}\leq r(n),$$
	then for $n<N_j$,
	$$R(n)=r^{(j)}(n)+r^{(j)}(n+N_j).$$
	From \eqref{eq:r(n)=1} and \eqref{eq:C(N_j)=O(1)} we have 
	\begin{align*}
		r(n)&=O(1),\\
		\sum_{n\geq0}r(n)(r(n)-1)&=O(1),\\
		\sum_{0\leq n<N_j}r^{(j)}(n+N_j)&=\mathcal{C}(N_j)+r(N_j)=O(1).
	\end{align*}
	Then
	\begin{align*}
		\sum_{0\leq n<N_j}R(n)(R(n)-1)
		& = \sum_{0\leq n<N_j}(r^{(j)}(n)+r^{(j)}(n+N_j))(r^{(j)}(n)+r^{(j)}(n+N_j)-1)\\
		& = \sum_{0\leq n<2N_j}r^{(j)}(n)(r^{(j)}(n)-1)+\sum_{0\leq n<N_j}2r^{(j)}(n)r^{(j)}(n+N_j)\\
		& = O(1).
	\end{align*}	
	Hence $\sum_nD(n)(D(n)-1)=O(1)$. Recall $A(N_j)B(N_j)=N_j+O(1)$, then
	\begin{align*}
		\sum_{\substack{0\leq n<N_j\\D(n)\neq1}}1
		& = N_j-\sum_{\substack{0\leq n<N_j\\D(n)=1}}1\\
		& = \sum_{0\leq n<N_j}D(n)+O(1)-\sum_{\substack{0\leq n<N_j\\D(n)=1}}1\\
		& = \sum_{\substack{0\leq n<N_j\\D(n)\geq2}}D(n)+O(1)\\
		& \leq \sum_{0\leq n<N_j}D(n)(D(n)-1)+O(1)=O(1),
	\end{align*}
	which means the number of the exceptional $n$ is $O(1)$. 
\end{proof}

For $\mu$ a nonzero finite positive Borel measure with compact support, define the left and right endpoints of its support by
$$\ell_\mu=\min\mathrm{supp}\, \mu,\ r_\mu=\max\mathrm{supp}\ \mu.$$

\begin{lemma}\label{lm:supp_convolution}
	Let $\mu, \nu$ be two nonzero finite positive Borel measures with compact supports. Then
	$$\ell_{\mu*\nu}=\ell_\mu+\ell_\nu,\, r_{\mu*\nu}=r_\mu+r_\nu.$$
\end{lemma}

\begin{proof}
	By the definition we have
	\begin{equation}\label{eq:left_1}
		\ell_{\mu*\nu}\geq \ell_\mu+\ell_\nu.
	\end{equation}

	For any $\varepsilon>0$, $\ell_\mu\in\mathrm{supp}\ \mu$ and $\ell_\nu\in\mathrm{supp}\ \nu$ imply
	$$\mu([\ell_\mu, \ell_\mu+\varepsilon))>0,\ \nu([\ell_\nu, \ell_\nu+\varepsilon))>0,$$
	from which we find
	$$(\mu*\nu)([\ell_\mu+\ell_\nu, \ell_\mu+\ell_\nu+2\varepsilon))\geq\mu([\ell_\mu, \ell_\mu+\varepsilon))\nu([\ell_\nu, \ell_\nu+\varepsilon))>0.$$
	That is, any small open neighbourhood of $\ell_\mu+\ell_\nu$ has a positive measure. Hence $\ell_\mu+\ell_\nu\in\mathrm{supp}(\mu*\nu)$, and
	\begin{equation}\label{eq:left_2}
		\ell_{\mu*\nu}\leq \ell_\mu+\ell_\nu.
	\end{equation}
	Combining this with \eqref{eq:left_1}, we have $\ell_{\mu*\nu}=\ell_\mu+\ell_\nu$. The case for $r_{\mu*\nu}$ follows similarly.
\end{proof}

Let $\mathcal{L}_E$ be {\it normalized Lebesgue measure} supported on the given Borel set $E$ of finite positive Lebesgue measure. Particularly, $\mathcal{L}_{[0, 1]}$ is the {\it standard Lebesgue measure} on $[0, 1]$. Let $\delta_x$ be the {\it Dirac measure} at the point $x$.

Let $\mu, \nu$ be two Borel probability measures supported on $[0, 1]$. $(\mu, \nu)$ is called a {\it complementary pair} of measures with respect to $\mathcal{L}_{[0, 1]}$ if
$$\mu*\nu=\mathcal{L}_{[0, 1]}.$$
Particularly, we have the {\it trivial} complementary pair $(\mu, \nu)=(\delta_0, \mathcal{L}_{[0, 1]})$ and $(\mathcal{L}_{[0, 1]}, \delta_0)$.

In 2014, Gabardo and Lai \cite[Thm 2.2]{GL2014} completely decided the structure of complementary pair with respect to $\mathcal{L}_{[0, 1]}$. We use their first two factors of the decomposition for Lemma \ref{lm:GL}. 

\begin{theorem}[Gabardo--Lai, 2014]
	Let $(\mu, \nu)$ be a non-trivial complementary pair of measures with respect to $\mathcal{L}_{[0, 1]}$, then there exist integers $M_1, M_2\geq2$ and Borel probability measures $\mu', \nu'$ supported in $[0, (M_1M_2)^{-1}]$ with $\mu'*\nu'=\mathcal{L}_{[0, \frac{1}{M_1M_2}]}$, such that either
	$$\mu=\Big(\frac{1}{M_1}\sum_{j=0}^{M_1-1}\delta_{\frac{j}{M_1}}\Big)*\mu',\ \nu=\Big(\frac{1}{M_2}\sum_{j=0}^{M_2-1}\delta_{\frac{j}{M_1M_2}}\Big)*\nu',$$
	or $\mu, \nu$ interchanged.
\end{theorem}

\begin{lemma}\label{lm:GL}
	Let $(\mu, \nu)$ be a non-trivial complementary pair of measures with respect to $\mathcal{L}_{[0, 1]}$, then there exists $t_1, t_2\in(0, 1)$ such that
	$$\max\Big\{\frac{\mu([0, t_1])\nu([0, t_1])}{t_1}, \frac{\mu([0, t_2])\nu([0, t_2])}{t_2}\Big\}\geq \frac{3}{2}.$$ 
\end{lemma}

\begin{proof}
	Without loss of generality, suppose 
	$$\mu=\Big(\frac{1}{M_1}\sum_{j=0}^{M_1-1}\delta_{\frac{j}{M_1}}\Big)*\mu',\ \nu=\Big(\frac{1}{M_2}\sum_{j=0}^{M_2-1}\delta_{\frac{j}{M_1M_2}}\Big)*\nu'$$
	for integers $M_1, M_2\geq2$. For convenience, denote
	$$\alpha=\frac{1}{M_1}\sum_{j=0}^{M_1-1}\delta_{\frac{j}{M_1}},$$
	$$\beta=\frac{1}{M_2}\sum_{j=0}^{M_2-1}\delta_{\frac{j}{M_1M_2}}.$$
	Recall that $\ell_\mu, r_\mu$ denote the left and right endpoints of $\mathrm{supp}\ \mu$. By definition we have
	\begin{align*}
		\ell_\alpha=0, & \ r_\alpha=1-\frac{1}{M_1},\\
		\ell_\beta=0, & \ r_\beta=\frac{1}{M_1}-\frac{1}{M_1M_2}.
	\end{align*}
	We apply Lemma \ref{lm:supp_convolution} to the convolution
	$$\mu=\alpha*\mu',\ \nu=\beta*\nu',\ \mu*\nu=\mathcal{L}_{[0, 1]},$$
	then $\ell_\mu=\ell_\nu=\ell_{\mu'}=\ell_{\nu'}=0$ and
	$$r_\mu=1-r_\nu,\ r_{\mu'}=\frac{1}{M_1}-r_\nu,\ r_{\nu'}=r_\nu-\frac{1}{M_1}+\frac{1}{M_1M_2}.$$
	By the decomposition, $\mu$ is just $M_1$ copies of $\mu'$ with each shrinking to $1/M_1$ of the original measure. More precisely, for all $0\leq j\leq M_1-1$ we have
	$$\mu\Big(\Big[\frac{j}{M_1}, \frac{j+1}{M_1}-r_\nu\Big]\Big)=\frac{1}{M_1},$$
	We choose the first two copies of $\mu'$ to get
	\begin{equation}\label{eq:mu(t_1)}
		\mu\Big(\Big[0, \frac{2}{M_1}-r_\nu\Big]\Big)\geq\frac{2}{M_1}.
	\end{equation}
	Notice that $\nu([0, r_\nu])=1$, and from $r_{\mu'}\geq0$ we know $2/M_1-r_\nu\geq r_\nu$, then
	\begin{equation}\label{eq:nu(t_1)}
		\nu\Big(\Big[0, \frac{2}{M_1}-r_\nu\Big]\Big)=\nu([0, r_\nu])=1.
	\end{equation}
	Let $t_1=r_{\mu'}+1/M_1=2/M_1-r_\nu$. From \eqref{eq:mu(t_1)}\eqref{eq:nu(t_1)} we have
	\begin{equation}
		\frac{\mu([0, t_1])\nu([0, t_1])}{t_1}\geq\frac{2}{2-M_1r_\nu}.
	\end{equation}
	Similarly, choose $t_2=r_{\nu'}+1/M_1M_2\geq r_{\mu'}$ and we have
	$$\mu([0, t_2])\geq\frac{1}{M_1},$$
	$$\nu([0, t_2])\geq\frac{2}{M_2},$$
	then
	\begin{equation}
		\frac{\mu([0, t_2])\nu([0, t_2])}{t_2}\geq\frac{2}{2-M_2+M_2M_1r_\nu}.
	\end{equation}
	Let $u=M_1r_\nu\in[1-1/M_2, 1]$, we have
	$$\max\Big\{\frac{\mu([0, t_1])\nu([0, t_1])}{t_1}, \frac{\mu([0, t_2])\nu([0, t_2])}{t_2}\Big\}\geq\max\Big\{\frac{2}{2-u}, \frac{2}{2-M_2+M_2u}\Big\}\geq\frac{2M_2+2}{M_2+2},$$
	then our result follows from $M_2\geq2$.
\end{proof}

\section{Proof of the main theorem}\label{section:Chen_proof}

Let $A, B$ be infinite almost perfect additive complements. As in Lemma \ref{lm:everything}, choose the constant $L$ and the infinite sequence $\{N_j\}$ such that $A(N_j)B(N_j)-N_j=L$. We translate the information into the language of measure theory. Define
\begin{equation}
	\mu_j:=\frac{1}{A(N_j)}\sum_{\substack{a\in A\\0\leq a\leq N_j}}\delta_{\frac{a}{N_j}}
\end{equation}
\begin{equation}
	\nu_j:=\frac{1}{B(N_j)}\sum_{\substack{b\in B\\0\leq b\leq N_j}}\delta_{\frac{b}{N_j}}
\end{equation}
be the Borel probability measures associated with $A, B$ on $[0, 1]$. Take any continuous function $f:[0, 2]\to\mathbb{R}$. Note that $f$ is bounded. 
\begin{align*}
	\int f\mathrm{d}(\mu_j*\nu_j)
	& = \iint f(x+y)\mathrm{d}\mu_j(x)\mathrm{d}\nu_j(y)\\
	& = \frac{1}{A(N_j)B(N_j)}\sum_{\substack{a\in A\\0\leq a\leq N_j}}\sum_{\substack{b\in B\\0\leq b\leq N_j}}f\Big(\frac{a+b}{N_j}\Big)\\
	& = \frac{1}{A(N_j)B(N_j)}\Big(\sum_{0\leq n\leq n_0}(r(n)-1)f\Big(\frac{n}{N_j}\Big)+\sum_{0\leq n\leq N_j}f\Big(\frac{n}{N_j}\Big)+O_f(\mathcal{C}(N_j))\Big)\\
	& = \frac{N_j}{N_j+L}\Big(\frac{1}{N_j}\sum_{0\leq n\leq N_j}f\Big(\frac{n}{N_j}\Big)\Big)+O(N_j^{-1})\\
	& \to \int_0^1f(t)\mathrm{d}t, j\to\infty.
\end{align*}
Thus we have the weak convergence
\begin{equation}\label{eq:weak_conv}
	\mu_j*\nu_j\to\mathcal{L}_{[0, 1]}.
\end{equation}
Since $[0, 1]$ is compact, by Prokhorov Theorem we can choose the subsequences $\{\mu_{j'}\}, \{\nu_{j'}\}$ of $\{\mu_j\}, \{\nu_j\}$ such that the weak convergence
\begin{equation}\label{eq:mu_weak_conv}
	\mu_{j'}\to\mu
\end{equation}
\begin{equation}\label{eq:nu_weak_conv}
	\nu_{j'}\to\nu
\end{equation}
holds for some Borel probability measures $\mu, \nu$ on $[0, 1]$ with $\mu*\nu=\mathcal{L}_{[0, 1]}$. We will keep using the symbol $\mu_j, \nu_j$ for convenience. 

\subsection{The trivial complementary pair}\

Without loss of generality, suppose $\mu=\delta_0, \nu=\mathcal{L}_{[0, 1]}$. A combinatorial method is available with \eqref{eq:a_j=o(N_j)} the only information from measures.

\begin{theorem}
	Following the former notations. If $\mu=\delta_0, \nu=\mathcal{L}_{[0, 1]}$, then we have
	\begin{equation}
		\limsup_{x\to\infty}\frac{A(x)B(x)}{x}=2.
	\end{equation}
\end{theorem}

\begin{proof}
	Since $A$ is infinite, the set $A\cap(N_j, \infty)$ is not empty. Let 
	$$a_j=\max(A\cap[0, N_j]),$$
	$$g_j=\min(A\cap(N_j, \infty))-1.$$
	Then we have $A(a_j)=A(N_j)=A(g_j), g_j+1\in A$. 

	Suppose there exists $\varepsilon>0$ such that $a_j>\varepsilon N_j$ for infinite $j$. Then $(a_j, b)$ is a crossing pair for any $b\in B\cap[(1-\varepsilon)N_j, N_j]$. Since $\nu_j\to\mathcal{L}_{[0, 1]}$, the number of such $b$ is $(\varepsilon+o(1))B(N_j)$, then $\mathcal{C}(N_j)\geq (\varepsilon+o(1))B(N_j)\to\infty$, which contradicts \eqref{eq:C(N_j)=O(1)}. Thus we have
	\begin{equation}\label{eq:a_j=o(N_j)}
		a_j=o(N_j).
	\end{equation}

	Recall the definition of crossing pairs, we have
	\begin{equation}
		A(g_j)B(g_j)=g_j+\mathcal{C}(g_j)+O(1),
	\end{equation}
	while \eqref{eq:A(u)B(v)<u+v} gives
	\begin{equation}
		A(g_j)B(g_j)=A(a_j)B(g_j)\leq a_j+g_j+O(1).
	\end{equation}
	Thus $\mathcal{C}(g_j)=O(a_j)$ and $A(g_j)B(g_j)=(1+o(1))g_j$.

	Choose a sequence $x_j\in(a_j, N_j)$ such that $a_j=o(x_j), x_j=o(N_j)$, for example we choose $x_j=[\sqrt{(a_j+1)N_j}]$. We claim and will prove later that
	\begin{equation}\label{eq:A_g}
		A(g_j+x_j)-A(g_j)\geq (1-o(1))A(g_j).
	\end{equation}
	\begin{equation}\label{eq:B_g}
		B(g_j+x_j)-B(g_j)=O(1).
	\end{equation}
	Now we have
	$$A(g_j+x_j)=A(g_j+x_j)-A(g_j)+A(g_j)\geq (2-o(1))A(g_j),$$
	$$B(g_j+x_j)=B(g_j+x_j)-B(g_j)+B(g_j)=(1+o(1))B(g_j),$$
	$$g_j+x_j=(1+o(1))g_j,$$
	$$A(g_j)B(g_j)=(1+o(1))g_j.$$
	Thus
	\begin{equation}\label{eq:nontrivial}
		\frac{A(g_j+x_j)B(g_j+x_j)}{g_j+x_j}\geq2-o(1).
	\end{equation}
	Finally, we only need to prove \eqref{eq:A_g}\eqref{eq:B_g}. 

	To show \eqref{eq:B_g}, for any $b\in B\cap[g_j+1, g_j+x_j]$, suppose there exists a pair $(a', b')\in (A\cap[0, N_j])\times(B\cap[0, N_j])$ such that 
	\begin{equation}\label{eq:congruence}
		b'-a'\equiv b-(g_j+1)\ (\mathrm{mod}\ N_j).
	\end{equation}
	In fact we have $a'\leq a_j=o(N_j)$ and
	$$-a_j\leq b'-a'\leq N_j-1,$$
	$$0\leq b-(g_j+1)\leq x_j-1.$$
	Since $a_j+x_j=o(N_j)$, for large $N_j$ the congruence \eqref{eq:congruence} becomes an equality, and
	\begin{equation}
		(g_j+1)+b'=a'+b.
	\end{equation}
	Thus there exist two different representations of a large integer, which contradicts \eqref{eq:r(n)=1}. Therefore, Lemma \ref{lm:unique_difference} implies that there are at most $O(1)$ such $b$ in $B\cap[g_j+1, g_j+x_j]$. 

	To show \eqref{eq:A_g}, by \eqref{eq:r(n)=1} every integer in $[g_j+1, g_j+x_j]$ has a unique representation for large $g_j$. Then
	\begin{align*}
		x_j
		& = \sum_{\substack{0\leq a, b\leq g_j\\g_j+1\leq a+b\leq g_j+x_j}}1+\sum_{\substack{g_j+1\leq a\leq g_j+x_j\\0\leq b\leq g_j+x_j-a}}1+\sum_{\substack{g_j+1\leq b\leq g_j+x_j\\0\leq a\leq g_j+x_j-b}}1\\
		& \leq \sum_{\substack{0\leq a, b\leq g_j\\g_j+1\leq a+b}}1+\sum_{g_j+1\leq a\leq g_j+x_j}\sum_{0\leq b\leq x_j}1+\sum_{g_j+1\leq b\leq g_j+x_j}\sum_{0\leq a\leq a_j}1\\
		& = \mathcal{C}(g_j)+(A(g_j+x_j)-A(g_j))B(x_j)+(B(g_j+x_j)-B(g_j))A(a_j)\\
		& = o(x_j)+(A(g_j+x_j)-A(g_j))B(x_j)+o(x_j).
	\end{align*}
	Hence
	\begin{equation}\label{eq:AB>x}
		(A(g_j+x_j)-A(g_j))B(x_j)\geq (1-o(1))x_j.
	\end{equation}
	On the other hand, recall that $A(a_j)=A(g_j)$, then from \eqref{eq:A(u)B(v)<u+v} we have
	\begin{equation}\label{eq:AB<x}
		A(g_j)B(x_j)=A(a_j)B(x_j)\leq a_j+x_j+O(1)=(1+o(1))x_j.
	\end{equation}
	Then \eqref{eq:A_g} could be derived from \eqref{eq:AB>x}\eqref{eq:AB<x}.
\end{proof}

\subsection{The non-trivial complementary pair}\

For any common {\it continuous point} $u\in(0, 1)$ (i.e. $\mu(\{u\})=\nu(\{u\})=0$) of the distribution functions of $\mu,\nu$, 
$$\frac{A(uN_j)}{A(N_j)}=\mu_j([0, u])\to\mu([0, u]),\ \frac{B(uN_j)}{B(N_j)}=\nu_j([0, u])\to\nu([0, u]).$$
Recall that $A(N_j)B(N_j)/N_j\to1$, then
\begin{align*}
	\frac{A(uN_j)B(uN_j)}{uN_j}\to\frac{\mu([0, u])\nu([0, u])}{u}.
\end{align*}

Suppose $(\mu, \nu)$ is a non-trivial complementary pair with respect to $\mathcal{L}_{[0, 1]}$. Note that the discontinuous points of any measure are at most countable. Take $t=t_1, t_2$ as in Lemma \ref{lm:GL}, then let $t+\varepsilon\to t$ along a sequence of common continuous points, we have
\begin{equation}\label{eq:trivial}
	\limsup_{x\to\infty}\frac{A(x)B(x)}{x}\geq\max\Big\{\frac{\mu([0, t_1])\nu([0, t_1])}{t_1}, \frac{\mu([0, t_2])\nu([0, t_2])}{t_2}\Big\}\geq \frac{3}{2}.
\end{equation}

Now the proof of Theorem \ref{thm:chen} is done by combining \eqref{eq:nontrivial}\eqref{eq:trivial}. 

\section{Perfect additive complements}\label{section:eg}

Recall that $A, B$ are called perfect additive complements, if $A, B$ are additive complements with $r(n)=1$ for all $n\in\mathbb{Z}_{\geq0}$. We will give an example of perfect additive complements. The mixed-radix construction was investigated by Ma \cite{Ma2022} and developed by Fang and S\'andor \cite{FS2025}. We independently discovered and studied the example.

\begin{definition}
	For the integer sequence $\{d_k\}_{k\geq1}$ with $d_k\geq2$ for all $k$, define
	\begin{align*}
		A
		& = \Big\{\sum_{j\geq0}\varepsilon_{2j+1}\prod_{i=1}^{2j}d_i\mid \varepsilon_{2j+1}\in\{0, 1, \ldots, d_{2j+1}-1\}, \varepsilon_{2j+1}=0\ \text{for all but finitely many}\ j\Big\}\\
		& = \Big\{\varepsilon_1+\varepsilon_3d_1d_2+\varepsilon_5d_1d_2d_3d_4+\ldots\Big\}.
	\end{align*}
	\begin{align*}
		B
		& = \Big\{\sum_{j\geq0}\varepsilon_{2j+2}\prod_{i=1}^{2j+1}d_i\mid \varepsilon_{2j+2}\in\{0, 1, \ldots, d_{2j+2}-1\}, \varepsilon_{2j+2}=0\ \text{for all but finitely many}\ j\Big\}\\
		& = \Big\{\varepsilon_2d_1+\varepsilon_4d_1d_2d_3+\varepsilon_6d_1d_2d_3d_4d_5+\ldots\Big\}.
	\end{align*}
	Then $A, B$ are additive complements. We call $A, B$ the additive complements associated with the sequence $\{d_k\}$. 
	
	Particularly, $A, B$ are almost perfect additive complements, since we have
	$$A(C_k)B(C_k)-C_k=1$$
	by taking $C_k=d_1d_2\cdots d_{2k}-1$. 
\end{definition}

\begin{theorem}[Ma, 2022]\label{thm:Ma}
	Given an integer sequence $\{d_k\}_{k\geq1}$ with $d_k\geq2$ for all $k$, the associated additive complements $A, B$ are infinite perfect additive complements. Furthermore, let
	$$D_k=\sum_{j=0}^{k-1}(-1)^j\Big(\prod_{i=0}^jd_{k-i}\Big)^{-1}=\frac{1}{d_k}-\frac{1}{d_kd_{k-1}}+\cdots+\frac{(-1)^{k-1}}{d_kd_{k-1}\cdots d_1},$$
	then we have
	\begin{equation}\label{eq:limsup_calc}
		\limsup_{x\to\infty}\frac{A(x)B(x)}{x}=\limsup_{k\to\infty}\frac{2}{1+D_k}.
	\end{equation}
	Particularly, 
	$$\limsup_{x\to\infty}\frac{A(x)B(x)}{x}=2\Longleftrightarrow\limsup d_k=\infty.$$
\end{theorem}

\begin{corollary}[Chen's constant 3/2 is sharp]
	The additive complements $A, B$ associated with the sequence $\{2,2,2,2,\ldots\}$ satisfy
	$$\limsup_{x\to\infty}\frac{A(x)B(x)}{x}=\frac{3}{2}.$$
\end{corollary}

\begin{theorem}[Fang--S\'andor, 2025]
	Suppose that $A, B$ are infinite perfect additive complements, then there exists an integer sequence $\{d_k\}, d_k\geq2$ such that $A, B$ are additive complements associated with $\{d_k\}$, or $A, B$ interchanged.
\end{theorem}

B\'ar\'any, Fang and S{\'a}ndor \cite{BFS2023} investigated the {\it spectrum} (all possible $\limsup_xA(x)B(x)/x$) of perfect additive complements. We can give a short proof of a special case.

\begin{theorem}\label{thm:perfect}
	For the integer sequence $\{d_k\}_{k\geq1}$ with $d_k\geq2$ for all $k$, let $A, B$ be the associated additive complements. If $\limsup d_k=G<\infty$, then
	$$\limsup_{x\to\infty}\frac{A(x)B(x)}{x}\in\Big[2-\frac{2}{G+2}, 2-\frac{1}{G}\Big],$$
	while the left (resp. right) endpoint corresponds to the perfect additive complements associated with $\{G, G, G, G,\ldots\}$ (resp. $\{G, 2, G, 2, \ldots\}$).
	Particularly, infinite perfect additive complements $A, B$ satisfy
	$$\limsup_{x\to\infty}\frac{A(x)B(x)}{x}\in\Big\{\frac{3}{2}\Big\}\cup\Big[\frac{8}{5}, 2\Big].$$
\end{theorem}
\begin{proof}
	Following the notations in Theorem \ref{thm:Ma}. We only need to show that
	\begin{equation}
		\liminf_{k\to\infty}D_k\in\Big[\frac{1}{2G-1}, \frac{1}{G+1}\Big].
	\end{equation}
	Let $\liminf_kD_k=s$. Notice that $0<D_k\leq 1/2, 2\leq d_k\leq G$ for large $k$, and
	$$D_k=\frac{1}{d_k}(1-D_{k-1}), k\geq2.$$
	Choose a subsequence $\{d_{k_j}\}$ such that $d_{k_j}=G$ for all $j$. Then
	$$s\leq\liminf_j\min\{D_{k_j}, D_{k_j-1}\}=\liminf_j\min\Big\{\frac{1}{G}(1-D_{k_j-1}), D_{k_j-1}\Big\}\leq \frac{1}{G+1}.$$
	On the other hand, we have
	$$D_k=\frac{1}{d_k}\Big(1-\frac{1-D_{k-2}}{d_{k-1}}\Big)\geq\frac{1}{G}\Big(1-\frac{1-D_{k-2}}{2}\Big)=\frac{1+D_{k-2}}{2G}.$$
	Taking the lower limit, we have $s\geq(1+s)/(2G)$. Hence
	\begin{equation}
		s\geq\frac{1}{2G-1}.
	\end{equation}
	A construction could be derived from the proof, and easily verified by \eqref{eq:limsup_calc}.
\end{proof}

\section{Remarks}

\subsection{The terminology}\

It is true, but not trivial that all {\it perfect} additive complements are {\it almost perfect} additive complements. The claim can be derived from the argument in Section \ref{section:eg}. 

There exist non-trivial infinite {\it almost perfect} additive complements which are {\it not perfect}. Let $m$ be an integer, and $C, D\subseteq\{0, 1, \ldots, m-1\}$ with $0\in C\cap D$ and the map 
\begin{align*}
	C\times D &\to \mathbb{Z}/m\mathbb{Z}\\
	(c, d)&\mapsto c+d\ \mathrm{mod}\ m
\end{align*}
is a bijection. Then for any infinite almost perfect additive complements $E, F$, the sets 
$$A=C+mE,\ B=D+mF$$
are also almost perfect. For example, consider $m=4, C=\{0, 2\}, D=\{0, 3\}$, then $1\not\in A+B$ and $A, B$ are not perfect. We can start the operation with $E, F$ any infinite perfect additive complements, and the examples satisfy $r(n)\leq 1$ for all $n$. We have not found the example of almost perfect additive complements with $r(n)\geq2$ for some $n$.

Note that {\it almost perfect} implies {\it eventually unique}, but the reverse remains open. As in Lemma \ref{lm:everything}, 
$$A(x)B(x)-x\not\to\infty\implies r(n)=1\ \text{for sufficiently large}\ n.$$
We are not sure whether the two conditions are equivalent. It is true for the case of infinite perfect additive complements; If one of $A, B$ is {\it finite}, the equivalence is trivial. However, we have not found a proof for {\it infinite} sets. 
\begin{question}\label{ques:equiv?}
	For infinite additive complements $A, B$ with $r(n)=1$ for all large $n$, is it true that $\liminf_n(A(n)B(n)-n)<\infty$?
\end{question}

\subsection{The main theorem}\

\eqref{eq:weak_conv} is a bridge from the combinatorial problem to the analysis problem. It relies on the relation
\begin{equation}\label{eq:r(n)=1_notes}
	r(n)=1, n>n_0;
\end{equation}
\begin{equation}\label{eq:C(N_j)=O(1)_notes}
	\mathcal{C}(N_j)=O(1),\ \text{for some subsequence}\ \{N_j\}.
\end{equation}
Recall that for additive complements $A, B$,
$$A(N)B(N)-N=\sum_{0\leq n\leq N}(r(n)-1)+\mathcal{C}(N)+1.$$
Hence $A(x)B(x)-x\not\to\infty$ is equivalent to \eqref{eq:r(n)=1_notes}\eqref{eq:C(N_j)=O(1)_notes} taken together. It is easy to find a counterexample to show that \eqref{eq:C(N_j)=O(1)_notes} does not imply \eqref{eq:r(n)=1_notes}. It remains open whether \eqref{eq:r(n)=1_notes} alone implies \eqref{eq:C(N_j)=O(1)_notes}, as shown in Question \ref{ques:equiv?}.

A strong result is obtainable with the full power of Gabardo--Lai theorem and a similar argument in Section \ref{section:Chen_proof}. We skip the details.

\begin{theorem}
	Fix the complementary pair $(\mu, \nu)$ given in \eqref{eq:mu_weak_conv}\eqref{eq:nu_weak_conv} and let $\{M_k\}$ be the integer sequence associated with the complementary pair \cite[Thm. 2.2]{GL2014}. 

	$(i)$ In the Type I (finite) case,
	$$\limsup_{x\to\infty}\frac{A(x)B(x)}{x}=2.$$

	$(ii)$ In the Type II (infinite) case,
	$$\limsup_{x\to\infty}\frac{A(x)B(x)}{x}\geq\frac{2M_k+2}{M_k+2}, k\geq2.$$
	Particularly, the upper limit is $\geq8/5$ if $M_k\geq3$ for some $k\geq2$, and equals $2$ for unbounded $\{M_k\}$.
\end{theorem}

In fact, we can get a better lower bound with more information on $\{M_k\}$. However, since we simply choose a specific subsequence $\{N_j\}$ and lose some information, the idea is only available to obtain the lower bound of the upper limit instead of the exact value. Furthermore, we are not able to rule out the following case: Consider the sequence $\{d_k\}$ defined by
$$d_k=
\begin{cases}
	3,\ k=n^2\ \text{for some integer}\ n,\\
	2,\ \text{otherwise}.
\end{cases}
$$
It is possible to choose a bad subsequence, for example $N_j=\prod_{k\leq j(j+1)-1}d_k-1$, such that the corresponding Gabardo--Lai sequence $\{M_k\}$ satisfies $M_k=2$ for all $k$. Then we are only able to get a lower bound of $3/2$ but no gaps on the right side of $3/2$, while the exact value for this example is $18/11>8/5$.

Another problem for us is to find a purely combinatorial method for the result. Thanks to Zhang \cite{Zh2026}, we know that the discrete proof does exist.

\subsection{Miscellaneous}\

In Gabardo and Lai's paper \cite{GL2014}, they gave a new proof of the classification of perfect additive complements for finite sets $\{0, 1, \ldots, N-1\}$ with their convolution decomposition theorem.

The second moment in the proof of Lemma \ref{lm:unique_difference} is sometimes called the {\it additive energy}, which is useful in some combinatorial problems.

B\'ar\'any, Fang and S{\'a}ndor treated the problem from a different point of view. Our Theorem \ref{thm:perfect} gives an elementary proof of the consequence of \cite[Lemma 4.3]{BFS2023}, and a similar result follows.
\begin{theorem}
	Let $A, B$ be additive complements associated with $\{d_k\}, d_k\in\mathbb{Z}_{\geq2}$. If
	$$\limsup_kd_k=G,\ \liminf_kd_k=g,$$
	with $2\leq g\leq G<\infty$, then
	$$\liminf_{k\to\infty}D_k\in\Big[\frac{g-1}{gG-1}, \frac{1}{G+1}\Big].$$
\end{theorem} 

Fix an integer $G\geq3$. Consider the spectrum with parameter $G$:
$$\mathfrak{L}_G:=\Big\{\limsup_{x\to\infty}\frac{A(x)B(x)}{x}\mid A, B\ \text{associated with}\ \{d_k\}, \limsup_kd_k=G\Big\}\subseteq\Big[2-\frac{2}{G+2}, 2-\frac{1}{G}\Big].$$
This will lead to the further problem concerning the structure. The following is an example and was partially solved by B\'ar\'any, Fang and S{\'a}ndor \cite{BFS2023}.
\begin{question}
	Decide all accumulation points of the spectrum $\mathfrak{L}_G$.
\end{question}

\subsection*{Acknowledgements}

We thank Jin-Hui Fang for the wonderful lecture in Xi'an and further discussion on this problem during July. We are also grateful to Yong-Gao Chen and Ping Xi for helpful comments on a previous version of the paper.

The second-named author is supported in part by Shaanxi NSF (No. 2025JC-QYCX-002) and Shaanxi Fundamental Science Research Project for Mathematics and Physics (No.25JSZ007).

\bibliography{ref}
\bibliographystyle{plain}

\end{document}